\pdfoutput=1
\documentclass{amsart}

\usepackage[utf8]{inputenc}
\usepackage{mathrsfs}
\usepackage{amsfonts, amssymb, amsmath, amsthm}
\usepackage{microtype}
\usepackage{hyperref}
\hypersetup{final, colorlinks=true, citecolor=blue}
\usepackage{url}
\usepackage{bbm}
\usepackage{enumitem}

\newcommand{\Q}{\mathbb{Q}}
\newcommand{\R}{\mathbb{R}}
\newcommand{\C}{\mathbb{C}}
\newcommand{\Z}{\mathbb{Z}}
\newcommand{\HH}{\mathcal{H}}
\newcommand{\cl}{\operatorname{Cl}}

\newcommand{\GL}{\operatorname{GL}}
\newcommand{\Bil}{\operatorname{Bil}}
\newcommand{\Sym}{\operatorname{Sym}}
\newcommand{\Hom}{\operatorname{Hom}}
\newcommand{\Stab}{\operatorname{Stab}}

\def\llangle{\langle\!\langle}
\def\rrangle{\rangle\!\rangle}
\def\llll#1{\left\llangle{#1}\right\rrangle}
\def\ll#1{{\left\langle{#1}\right\rangle}}
\def\id#1{{\mathfrak{#1}}}
\def\idn#1{{\norm(\mathfrak{#1})}}
\def\hatx#1{{{\widehat{#1}}^\times}}
\newcommand{\p}{_{\id p}}
\newcommand{\bbu}{\mathbbm{1}}

\def\abs#1{\left\vert{#1}\right\vert}

\DeclareMathOperator{\norm}{\mathcal{N}}
\DeclareMathOperator{\trace}{\mathcal{T}}
\DeclareMathOperator{\disc}{\Delta}
\newcommand{\OO}[1][]{\mathcal{O}_{\mkern-2mu #1}}
\newcommand{\OOx}[1][]{\mathcal{O}_{\mkern-3mu #1}^{\times}}
\newcommand{\hatOO}[1][]{\widehat{\mathcal{O}}_{\mkern-2mu #1}}
\newcommand{\hatOOx}[1][]{\hatx{\mathcal{O}}_{\mkern-3mu #1}}
\def\dF{d_{\mkern-2mu F}}
\def\IF{\mathcal{J}_{\mkern-2mu F}}
\def\hF{h_{\mkern-2mu F}}
\def\tK{t_{\mkern-2mu K}}
\def\mK{m_{\mkern-2mu K}}
\def\muR{\mu_{\mkern-2mu R}}
\def\Tm{\mathop{T_{\mkern-1mu \id{m}}}}
\def\Tg{\mathop{T_{\mkern-2mu g}}}

\def\sC{\mathcal C}
\def\DC{\mathcal D(\sC)}
\def\DK{{\mathfrak{D}_{\!K}}}

\numberwithin{equation}{section}

\theoremstyle{plain}
\newtheorem{prop}[equation]{Proposition}
\newtheorem{thm}[equation]{Theorem}
\newtheorem*{thm*}{Theorem}

\newtheorem{coro}[equation]{Corollary}
\newtheorem{lemma}[equation]{Lemma}

\theoremstyle{remark}

\newtheorem{rmk}[equation]{Remark}

\begin{document}

\title
[An explicit Waldspurger formula for Hilbert modular forms]
{An explicit Waldspurger formula\\ for Hilbert modular forms}

\author{Nicol\'as Sirolli}
\address{Universidad de la República, Montevideo, Uruguay}
\email{nsirolli@dm.uba.ar}
\thanks{The first author was fully supported by the
    ANII
    under grant code PD\_NAC\_2013\_1\_11010}

\author{Gonzalo Tornar\'ia}
\address{Universidad de la República, Montevideo, Uruguay}
\email{tornaria@cmat.edu.uy}


\begin{abstract} 

We describe a construction of preimages for the Shimura map on Hilbert modular 
forms, and give an explicit  Waldspurger type formula relating their Fourier 
coefficients to central values of twisted $L$-functions. Our construction is 
inspired by that of Gross and applies to any nontrivial level and arbitrary base 
field, subject to certain conditions on the Atkin-Lehner eigenvalues and on the 
weight.

\end{abstract}

\maketitle

\vspace{-3.5pt}
\section*{Introduction}

Computing central values of $L$-functions attached to modular forms is interesting 
because of the arithmetic information they encode. These values are related to 
Fourier coefficients of half-integral weight modular forms and the Shimura 
correspondence, as shown in great generality in \cite{waldspu-coeffs}. For 
classical modular forms, explicit formulas of Waldspurger type can be found in 
\cite{gross}, \cite{BSP}, \cite{mrvt}, among other works. In the Hilbert setting there are 
Waldspurger type formulas available in \cite{shim-coef}, \cite{baruch-mao}. More 
explicit formulas for computing central values in terms of Fourier coefficients 
can be found in \cite{hiraga-ikeda} in the case of trivial level, and in 
\cite{Xue-Waldspu} where the result is restricted to modular forms of prime 
power level over fields with odd class number. In \cite{caishutian} the authors 
give a formula in terms of heights, in the case of parallel weight $\mathbf{2}$.

In this article we prove a formula relating central values of twisted 
$L$-functions attached to a Hilbert cuspidal newform $g$ to Fourier coefficients 
of certain modular forms of half-integral weight, which are constructed 
explicitly as theta series and map to $g$ under the Shimura correspondence. Our 
result applies to any nontrivial level and arbitrary base field, and to a 
broader family of twists that the one considered in \cite{Xue-Waldspu}.
In the classical case it is more general than \cite{gross} and 
\cite{BSP}, where the authors consider prime and square-free 
levels respectively.

\medskip

Let $g$ be a normalized Hilbert cuspidal newform over a totally real number 
field $F$, of level $\id N \subsetneq \OO[F]$, weight $\mathbf{2}$ + 
2$\mathbf{k}$ and trivial central character. For each $\id p \mid \id N$ denote 
by $\varepsilon_g(\id p)$ the eigenvalue of the $\id p$-th Atkin-Lehner 
involution acting on $g$, and let $\mathcal{W}^{-} = \{\id p \mid \id 
N\,:\,\varepsilon_g(\id p) = -1\}$. We make the following hypotheses on 
$\mathcal{W}^-$ and $\mathbf{k}$.

\smallskip
\begin{enumerate}[label=\textbf{H\arabic*.},ref={H\arabic*}]
 \item $\vert\mathcal{W}^- \vert$ and $[F:\Q]$ have the same 
parity.\label{hyp:paridad}
 \item  $v\p(\id N)$ is odd for every $\id p \in \mathcal{W}^-$.\label{hyp:JL}
 \item $(-1)^{\mathbf{k}} = 1$. 
\label{hyp:peso}
\end{enumerate}
\smallskip

For $D\in F^+$ denote by $L_D(s,g) = L(s,g) L(s,g\otimes \chi_D)$ the 
Rankin-Selberg convolution \textit{L}-function of $g$ by the quadratic character 
$\chi_D$ associated to the extension $F(\sqrt{-D})/F$, normalized with center of 
symmetry at $s=1/2$. The main result of this article is stated in 
Theorem~\ref{thm:main_thm}; in the simpler form given by Corollary~\ref{coro:main} it
claims that there exists a 
Hilbert cuspidal form $f$ of weight $\mathbf{3/2+k}$ whose Fourier coefficients 
$\lambda(D,\id a;f)$ satisfy
 \[
  L_D(1/2,g)= \frac{c_D}{D^\mathbf{k+1/2}} \,\vert \lambda(D,\id 
a;f)\vert^2 \,,
 \]
for every $D$ such that
$\chi_D(\id p)=\varepsilon_g(\id p)$
whenever $v\p(\id N)$ is odd
and such that
the conductor of $\chi_D$ is prime to $2\id N$.
Here  $c_D$ and $\id a$ are, respectively, an 
explicit positive rational number and a fractional ideal of $F$, both depending 
only on $D$.
If $L(1/2,g) \neq 0$ then $f \neq 0$ and it maps to $g$ under the Shimura 
correspondence. Actually, in this case we construct a linearly independent 
family of preimages for the Shimura correspondence, as shown in Corollary 
\ref{coro:family}.

A different generalization of Gross's formula in \cite{mrvt,mao}
gets rid of the restriction on $D$ for classical
modular forms of prime level.  In future work we will combine this
idea with our methods to obtain a formula without
restrictions on $D$.

\medskip
This article is organized as follows.
In Section~\ref{sect:quat_forms} we state some basic 
facts about the space of quaternionic modular forms.
In Section~\ref{sect:half_int} we 
show how to obtain half-integral weight Hilbert modular forms out of 
quaternionic modular forms, and give a formula for their Fourier coefficients in terms 
of special points and the height pairing on the space of quaternionic modular forms.
In Section~\ref{sect:height_geom}
we relate central values of twisted $L$-functions
to the height pairing,
using results of \cite{zhang-gl2} and \cite{Xue-Rankin}
about central values of Rankin $L$-functions.
In Section~\ref{sect:orders} we state an auxiliary result, needed 
for the proof of the main result of this article, which we give in
Section~\ref{sect:main}.

\subsection*{Acknowledgements} The first author would like to thank
both IMERL and CMAT, Universidad de la República, for hosting him during
the postdoctoral stay in which this article was written.

\subsection*{Notation summary}

We fix a be a totally real number field $F$ of discriminant $\dF$,
with ring of integers $\OO[F]$. 
We denote by $\IF$ the group of fractional ideals of $F$,
and we write $\cl(F)$ for the class group, and $\hF$ for the class
number.
We denote by $\mathbf{a}$ the set of embeddings $\tau : F \hookrightarrow \R$, 
and we let $F^+ = \{\xi \in F\,:\,\tau(\xi) > 0 \; \forall \tau \in 
\mathbf{a}\}$. Given $\mathbf{k} = (k_\tau) \in \Z^\mathbf{a}$ and $\xi \in F$, 
we let $\xi^\mathbf{k} = \prod_{\tau \in \mathbf{a}} \tau(\xi)^{k_\tau}$. By 
$\id p$ we al\-ways denote a prime ideal of $\OO[F]$, and we use $\id p$ as a 
subindex to denote completions of global objects at $\id p$. 
Given an integral ideal $\id N \subseteq \OO[F]$ we let $\omega(\id N) = \left\vert 
\left\{\id p\,:\,\id p \mid \id N\right\} \right\vert$.
Given $\id p$ we denote by $\pi\p$ a local uniformizer at $\id p$, and we let 
$v\p$ denote the $\id p$-adic valuation.

Given a totally imaginary quadratic extension $K/F$ we let $\OO[K]$ be
the maximal order, and let $\DK \subseteq \OO[F]$
denote the relative discriminant. We let $\tK = [\OOx[K]:\OOx[F]]$, 
and let $\mK \in \{1,2\}$ be the order of the kernel of the natural map 
$\cl(F)\to\cl(K)$.

Given a quaternion algebra $B/F$ we denote by $\norm:B^\times\to F^\times$
and $\trace:B\to F$ the reduced norm and trace maps, and we use 
$\norm$ and $\trace$ to denote other norms and traces as well. We 
denote by $\widehat B = \prod'_{\id p} B\p$ and $\hatx B = \prod'_{\id p} 
B\p^\times$ the corresponding restricted products, and we let $B_\infty = 
\prod_{\tau \in \mathbf{a}} B_\tau$. Finally, given a level $\id N \subseteq 
\OO[F]$, an integral or half-integral weight $\mathbf k$ and a Hecke character 
$\chi$, we denote by $\mathcal{M}_{\mathbf{k}}(\id N, \chi)$ and 
$\mathcal{S}_{\mathbf{k}}(\id N, \chi)$ the corresponding spaces of Hilbert 
modular and cuspidal forms.

\section{Quaternionic modular forms}\label{sect:quat_forms}

Let $B$ be a totally definite quaternion algebra over $F$. Let $(V, \rho)$ be 
an irreducible unitary right representation of $B^\times/F^\times$, which we 
denote by $(v,\gamma)\mapsto v\cdot \gamma$. Let $R$ be an order of (reduced) 
discriminant $\id N$ in $B$. A \emph{quaternionic modular form} of weight 
$\rho$ and level $R$ is a function $\varphi: \hatx B\to V$ such that for every 
$x\in \hatx B$ the following transformation formula is satisfied:
\[
 \varphi(u x \gamma) = \varphi(x) \cdot \gamma \qquad
 \forall\,u\in\hatx R,\,\gamma\in B^\times\,.
\]
The space of all such functions is denoted by $\mathcal{M}_\rho(R)$. We let 
$\mathcal{E}_\rho(R)$ be the subspace of functions that factor through the map 
$\norm:\hatx B \to \hatx F$. These spaces come equipped with the action 
of Hecke operators $\Tm$, indexed by integral ideals $\id m \subseteq 
\OO[F]$, and given by
\[
 \Tm\varphi(x) = \sum_{h \in \hatx R \backslash H_\id{m}} \varphi(h x) 
\,,
\]
where $H_\id{m} = \left\{h \in \widehat R \, : \, \hatOO[F]
\norm(h) \cap \OO[F] = \id m\right\}$.

Given $x \in \hatx B$, we let
\[
 \widehat R_x = x ^{-1} \widehat R \, x , \qquad R_x = B\cap \widehat R_x, 
 \qquad \Gamma_{\!x} = R_x^\times / \OOx[F],
 \qquad w_x = \vert \Gamma_{\!x} \vert\,.
\]
The sets $\Gamma_{\!x}$ are finite since $B$ is totally definite. Let $\cl(R) = 
\hatx R \backslash \hatx B / B^\times$. We define an inner product on 
$\mathcal{M}_\rho(R)$, called the \emph{height pairing}, by
\[
 \ll{\varphi,\psi} = \sum_{x\in \cl(R)} \tfrac{1}{w_x}
 \ll{\varphi(x),\psi(x)}\,. 
\]
The space of \emph{cuspidal forms} $\mathcal{S}_\rho(R)$ is defined as the 
orthogonal complement of $\mathcal{E}_\rho(R)$ with respect to this pairing.

Let $N(\widehat R) = \{z \in \hatx B\,:\,\widehat R_z = \widehat R\}$ be the 
normalizer of $\widehat R$ in $\hatx B$.  We let $\widetilde \Bil(R) = \hatx 
R \backslash N(\widehat R) / F^\times$. We have an embedding $\cl(F) 
\hookrightarrow \widetilde \Bil(R)$. The group $\widetilde \Bil(R)$, and in 
particular $\cl(F)$, acts on $\mathcal{M}_\rho(R)$ by letting $(\varphi\cdot 
z)(x)=\varphi (zx)$. This action restricts to $\mathcal{S}_\rho(R)$, and is 
related to the height pairing by the equality 
\begin{equation}\notag
\ll{\varphi \cdot z,\psi \cdot z} = \ll{\varphi, \psi}\,.
\end{equation}
The action of $\widetilde \Bil(R)$ commutes with the action of the Hecke 
operators. The adjoint of $\Tm$ with respect to the height pairing is 
given by $\varphi \mapsto \Tm\varphi \cdot \id{m}^{-1}$.

The subspaces of $\mathcal{M}_\rho(R)$ and $\mathcal{S}_\rho(R)$ fixed by the 
action of $\cl(F)$ are denoted by $\mathcal{M}_\rho(R,\bbu)$ and 
$\mathcal{S}_\rho(R,\bbu)$. Let $\Bil(R) = \hatx R \backslash N(\widehat R) / 
\hatx F$. Then $\Bil(R)$ acts on $\mathcal{M}_\rho(R,\bbu)$ and 
$\mathcal{S}_\rho(R,\bbu)$, and we have a short exact sequence 
\begin{equation}\label{eqn:s.e.s.}
 1 \longrightarrow \cl(F) \longrightarrow \widetilde \Bil(R) \longrightarrow 
\Bil(R) \longrightarrow 1\,.
\end{equation}

\subsection*{Forms with minimal support}
Given $x\in \hatx B$ and $v\in V$, let $\varphi_{x,v} \in\mathcal{M} 
_\rho(R)$ be the quaternionic modular form given by
\[
 \varphi_{x,v}(y) = \sum_{\gamma\in \Gamma_{\!x,y}} v \cdot \gamma\,,
\]
where $\Gamma_{\!x, y} = (B^\times \cap x ^{-1} \hatx R y) / \OOx[F]$. 
Note that $\varphi_{x,v}$ is supported in $\hatx R x B^\times$. Furthermore, we 
have that 
\begin{alignat}{2}
    \varphi_{ux\gamma, v} & = \varphi_{x,v\cdot\gamma^{-1}}
    && \forall \, u\in\hatx R, \, \gamma \in B^\times,
    \label{eqn:transf_phi2}
    \\
    \varphi_{x,v} \cdot z & = \varphi_{z^{-1} x, v}
    &\qquad& \forall \, z \in \widetilde\Bil(R)\,.
    \label{eqn:transf_phi}
\end{alignat}
Given $\varphi \in \mathcal{M}_\rho(R)$, using that $\varphi(x) \in 
V^{\Gamma_{\!x}}$ for every $x \in \hatx B$ we get that
\begin{equation}\label{eqn:phi=sum_phix}
 \varphi = \sum_{x\in \cl(R)} \tfrac{1}{w_x} \varphi_{x,\varphi(x)}\,.
\end{equation}

\begin{prop}\label{prop:hecke_on_phi}
 
Let $x\in \hatx B$ and $v\in V$. Then $\Tm\varphi_{x,v} = \sum_{h\in 
H_\id{m} / \hatx R}  \varphi_{h^{-1}x,v}$\,.
 
\end{prop}

\begin{proof}
 
Given $y\in \hatx B$, let $\Gamma = (B^\times \cap x^{-1} H_\id{m}\,y) / 
\OOx[F]$. Then
\[
 \Gamma = \coprod_{h \in \hatx R \backslash H_\id{m}}\Gamma_{\!x,hy}
 = \coprod_{h\in H_\id{m} / \hatx R}  \Gamma_{\!h^{-1}x,y}\,.
\]
The first decomposition implies that
\begin{align*}
 \Tm\varphi_{x,v} (y) 
 = \sum_{h \in \hatx R \backslash H_\id{m}} \varphi_{x,v}(h y)
 = \sum_{h \in \hatx R \backslash H_\id{m}} \sum_{\gamma \in \Gamma_{\!x,hy}} v 
\cdot \gamma = \sum_{\beta \in \Gamma} v\cdot \beta\,,
\intertext{whereas the second decomposition implies that}
 \sum_{h\in H_\id{m} / \hatx R} \varphi_{h^{-1}x,v} (y)
 = \sum_{h \in H_\id{m} / \hatx R} \sum_{\gamma \in \Gamma_{\!h^{-1}x,y}} v 
\cdot \gamma = \sum_{\beta \in \Gamma} v\cdot \beta\,,
\end{align*}
which completes the proof.
\end{proof}

\begin{prop}\label{prop:height_on_phi}
Given $x,y\in\hatx B$ and $v,w\in V$, we have
\[
 \ll{\varphi_{x,v}\,, \varphi_{y,w}}=
\sum_{\gamma\in\Gamma_{\!x,y}} 
\ll{v\cdot\gamma, w}.
\]
\end{prop}
\begin{proof}
Since $\Gamma_{\!y}=\Gamma_{\!y,y}$ acts on $\Gamma_{\!x,y}$ on the right, using that $(\rho, 
V)$ is unitary we get that
\begin{align*}
\ll{\varphi_{x,v}, \varphi_{y,w}}
&=
  \tfrac{1}{w_y}
  \ll{\varphi_{x,v}(y),\varphi_{y,w}(y)}
=
  \tfrac{1}{w_y}
  \sum_{\alpha \in \Gamma_{\!x,y}}
  \sum_{\beta \in \Gamma_{\!y}} 
    \ll{v \cdot \alpha, w \cdot \beta} 
\\&=
  \tfrac{1}{w_y}
  \sum_{\alpha \in \Gamma_{\!x,y}}
  \sum_{\beta \in \Gamma_{\!y}} 
    \ll{v \cdot \alpha\beta^{-1}, w} 
=
\sum_{\gamma \in \Gamma_{\!x,y}}
    \ll{v \cdot \gamma, w} 
\,.
\qedhere
\end{align*}
\end{proof}

\section{Half-integral weight modular forms and special 
points}\label{sect:half_int}

From now on we specify the representation $(V,\rho)$. Let $\mathbf{k} =(k_\tau) 
\in \Z_{\geq 0}^\mathbf{a}$\,. For each $\tau\in\mathbf{a}$ we consider the 
real vector space $W_\tau = B_\tau / F_\tau$, with inner product induced by the 
totally positive definite quadratic form $-\disc(x) = 4\norm(x) - \trace(x)^2$.
By letting $B_\tau^\times / F_\tau^\times$ act on $W_\tau$ 
by conjugation we get an orthogonal representation. This gives naturally an 
orthogonal representation of $B_\tau^\times / F_\tau^\times$ on 
$\R_{k_\tau}[W_\tau] = \Sym^{k_\tau}(\Hom_\R(W_\tau,\R))$, the space of 
homogeneous polynomials on $W_\tau$ of degree $k_\tau$ with coefficients in 
$\R$, and hence a unitary representation of $B_\tau^\times / F_\tau^\times$ on 
$\C_{k_\tau}[W_\tau] = \R_{k_\tau}[W_\tau] \otimes_\R \C$. We let $V_{k_\tau}$ 
denote the $B_\tau^\times / F_\tau^\times$-submodule of $\C_{k_\tau}[W_\tau]$ of 
harmonic polynomials with respect to $-\disc$. This is, up to isomorphism, 
the unique irreducible unitary representation of $B_\tau^\times / 
F_\tau^\times$ of dimension $2k_\tau +1$. We let $V_\mathbf{k} = 
\otimes_{\tau\in\mathbf{a}} V_{k_\tau}$, and through the embedding $B^\times 
\hookrightarrow B_\infty^\times$ we get an irreducible unitary representation 
$(V_\mathbf{k}, \rho_\mathbf{k})$ of $B^\times / F^\times$. We denote the 
corresponding spaces of quaternionic modular forms by 
$\mathcal{M}_\mathbf{k}(R)$, etc.

Denote by $\HH$ the complex upper half-plane. Let $e_F:F\times \HH^\mathbf{a} 
\to \C$ be the exponential function given by $e_F(\xi, z)=\exp{\big(2\pi 
i\sum_{\tau\in\mathbf{a}} \tau(\xi)\, z_\tau\big)}$. Given $x\in \hatx B$, let 
$L_x \subseteq B/F$ be the lattice given by $L_x = R_x / \OO[F]$. Given $P \in 
V_\mathbf{k}$, we let $\vartheta_{x,P}:\HH^\mathbf{a}\to \C$ be the function 
given by
\[
 \vartheta_{x,P}(z)=\sum_{y\in L_x} P(y)\, e_F(-\disc(y),z/2)\,.
\]
These theta series satisfy
\begin{align}
 \vartheta_{x,P} & = (-1)^\mathbf{k} \, \vartheta_{x,P}\,, 
\label{eqn:theta_paridad}\\
\qquad\quad
 \vartheta_{z x\gamma,P\cdot \gamma} & = \vartheta_{x,P}\,,
\qquad\quad
\forall z\in N(\widehat R),\, \gamma\in B^\times. \label{eqn:theta_transf}
\end{align}

The following two propositions extend the results about theta series from 
\cite[Section 4]{yo} to arbitrary weights.

\begin{prop}
Let $x\in \hatx B$ and let $P\in V_\mathbf{k}$. Then $\vartheta_{x,P} \in 
\mathcal{M}_\mathbf{3/2+k}(4\id{N},\chi_1)$, where $\chi_1$ is the Hecke 
character associated to the extension $F(\sqrt{-1})/F$. Furthermore, for $D\in 
F^+\cup\{0\}$ and $\id a \in \IF$, the $D$-th Fourier coefficient of 
$\vartheta_{x,P}$ at the cusp $\id a$ is given by
 \begin{equation}\label{eqn:fourier_coeffs}
  \lambda(D, \id a; \vartheta_{x,P}) = \frac{1}{\idn a} \sum_{y \in 
\mathcal{A}_{D,\id a}(x)} P(y)\,,
 \end{equation}
 where $\mathcal{A}_{D,\id a}(x) = \{y \in \id a^{-1}L_x\,:\, \disc(y) = -D\}$.
\end{prop}

\begin{proof}

Let $\mathcal{B} = \{v_1,v_2,v_3\}$ be a basis of $B/F$ such that there exist 
$\id{a}_1$, $\id{a}_2$, $\id{a}_3 \in \IF$ satisfying 
$L_x=\oplus_{i=1}^3 \id{a}_i\,v_i$\,. Let $S \in \GL_3(F)$ be the matrix of the 
quadratic form $-\disc$ with respect to this basis. For each $\tau\in\mathbf{a}$
write $\tau(S) = A_\tau^t\,A_\tau$, with $A_\tau\in\GL_3(\R)$. Then there 
exist homogeneous harmonic polynomials $Q_\tau(X)$ of degree $k_\tau$ such that 
 \[
  P(y) = \prod_{\tau\in\mathbf{a}} Q_\tau(A_\tau\,\tau([y]_\mathcal{B}))\,,
 \]
where we denote by $[y]_\mathcal{B}$ the coordinates of $y$ with respect to the 
basis $\mathcal{B}$. For each $\tau$ we may assume that $Q_\tau(X) = (z_\tau^t 
\,X)^{k_\tau}$, with $z_\tau \in \C^3$ such that $z_\tau^t \,z_\tau = 0$ (see 
\cite[Theorem 9.1]{iwaniec-topics}). Let $\rho_\tau = (A_\tau^{-1}\, 
z_\tau)^t$. 
Then we have that $\rho_\tau^t\,\tau(S)\,\rho_\tau = 0$. Let $\sigma:F^3\to\C$ 
be the function given by $\sigma(\xi) = \prod_{\tau\in\mathbf{a}} 
(\rho_\tau^t\,\tau(S)\,\tau(\xi))^{k_\tau}$,
so that $P(y) = \sigma([y]_\mathcal{B})$,
and let $\eta:F^3\to\C$ be the 
characteristic function of $\id{a}_1\oplus\id{a}_2\oplus\id{a}_3$. Then we have 
that
 \[
 \vartheta_{x,P}(z) = 
 \sum_{\xi\in F^3}\eta(\xi)\,\sigma(\xi)\,e_F(\xi^t\,S\,\xi,z/2)\,.
 \]
The modularity of $\vartheta_{x,P} $ follows applying \cite[Proposition 
11.8]{shim-hh}. Finally, \eqref{eqn:fourier_coeffs} follows 
as in the case $\mathbf{k}=\mathbf{0}$ considered in \cite[Proposition 4.4]{yo}.
\end{proof}

The theta series $\vartheta_{x,P}$ defines a linear map 
$\vartheta:\mathcal M_\mathbf{k}(R)\to 
\mathcal{M}_\mathbf{3/2+k}(4\id{N},\chi_1)$, given by $\vartheta(\varphi_{x,P}) 
= \vartheta_{x,P}$. This map is well defined by \eqref{eqn:phi=sum_phix} and 
\eqref{eqn:theta_transf}, and satisfies
\[
 \vartheta(\varphi) = \sum_{x\in \cl(R)} \tfrac{1}{w_x} 
\vartheta_{x,\varphi(x)}\,.
\]
Note that if Hypothesis \ref{hyp:peso} does not hold, then $\vartheta = 0$ by 
\eqref{eqn:theta_paridad}.
 
\begin{prop}
 The map $\vartheta$ is Hecke-linear, and satisfies that
 \begin{equation*}
  \vartheta(\varphi\cdot z) = \vartheta(\varphi) \quad \forall\, z \in 
\widetilde \Bil(R)\,.
 \end{equation*}
 Furthermore, $\vartheta(\varphi)$ is cuspidal if and only if $\varphi$ is 
cuspidal.
\end{prop}

\begin{proof}
For $z \in \widetilde \Bil(R)$ we have, by \eqref{eqn:transf_phi} and 
\eqref{eqn:theta_transf}, that
\begin{equation*}
    \vartheta(\varphi_{x,P}\cdot z) =
    \vartheta(\varphi_{z^{-1} x,P}) =
    \vartheta(\varphi_{x,P})\,.
\end{equation*}
The Hecke-linearity was proved in \cite[Theorem 4.11]{yo} when $\mathbf{k} = 
\mathbf{0}$, and can be proved in the general case following the same lines.
The assertion about the cuspidality was proved in \cite[Theorem 4.11]{yo} when 
$\mathbf{k}=\mathbf{0}$, and in the remaining cases it follows from the facts that
$\mathcal M_\mathbf{k}(R) = \mathcal S_\mathbf{k}(R)$ and that every theta 
series is cuspidal.
\end{proof}

From now on $D$ denotes an element in $F^+$, and we denote $K = F(\sqrt{-D})$. 
Assume that there exists an embedding $K \hookrightarrow B$, which we fix. Let 
$P_D \in V_\mathbf{k}$ be the polynomial characterized by the property
\begin{equation}\label{eqn:gegenbauer}
  P (\omega) = \ll{P, P_D} \qquad \forall \, P\in V_\mathbf{k}\,,
\end{equation}
where $\omega \in K/F$ is such that $\disc(\omega) = -D$. Note that $\omega$
is uniquely determined up to sign. By Hypothesis \ref{hyp:peso} we have 
$P(-\omega) = P(\omega)$ for every $P \in V_\mathbf{k}$, which implies that 
$P_D$ does not depend on $\omega$. Since $(P \cdot a) (\omega) = P (\omega)$ for 
every $a\in K^\times$, we have that
\begin{equation}\label{eqn:P_D-Kinv}
 P_D \cdot a = P_D \quad \forall \, a \in K^\times / F^\times. 
\end{equation}

\begin{prop}
The polynomial $P_D$ satisfies
\begin{equation}\label{eqn:Mv(1)}
 \ll{P_D, P_D} = D^\mathbf{k}\,\prod_{\tau \in \mathbf{a}} s_{k_\tau}\,,
\end{equation}
where for $k\in \Z_{\geq 0}$ we denote by $s_k$ the positive rational number 
given by
\begin{equation}\label{eqn:s_k}
 s_k = \frac{1}{\Gamma(k+1/2)}\,
 \sum_{q=0}^{\left\lfloor \frac{k}{2} \right\rfloor} 
\:\frac{\Gamma(k+1/2-q)}{q!\,(k-2q)!\,2^{2q}}\,.
\end{equation}
\end{prop}

\begin{proof}
We have that $P_D = \otimes_{\tau\in \mathbf{a}} P_{D,\tau}$, where 
$P_{D,\tau}\in V_{k_\tau}$ is the polynomial characterized by the property
\[
 P_\tau (\omega) = \ll{P_\tau, P_{D,\tau}} \qquad \forall \, P_\tau\in 
V_{k_\tau}\,.
\]
 
Identifying $B_\tau$ with Hamilton quaternions $\ll{1,i,j,ij}_\R$ and letting 
$X_1 = i/2$, $X_2 = j/2$, $X_3 = ij/2$, we have that $\{X_1,X_2,X_3\}$ is an 
orthonormal basis for $W_\tau$ with respect to $-\disc$. Then the monomials 
$X_1^a\,X_2^b\,X_3^{k_\tau-a-b} \in \C_{k_\tau}[W_\tau]$ are 
orthogonal, and have norm equal to $a!\,b!\,(k_\tau-a-b)!$, which implies that 
the inner product $\ll{\:,\:}$ we consider on $V_{k_\tau}$ is related to the 
inner product $\llll{\:,\:}$ considered in \cite[Section 4.1]{gebh} by 
$\ll{\:,\:} = k_\tau! \, \llll{\:,\:}$. Hence \eqref{eqn:Mv(1)} follows from the 
explicit formulas for the Gegenbauer polynomials given in \cite[Proposition 
4.1.9]{gebh}, which imply that
\[
 \llll{P_{D,\tau},P_{D,\tau}} = \frac{\tau(D)^{k_\tau} s_{k_\tau}}{k_\tau!}\,.
 \qedhere
\]
\end{proof}

Given $\id a \in \IF$, we say that the pair $(-D,\id a)$ is a 
\emph{discriminant} if there exists $\omega \in K$ with $\disc(\omega) = -D$ 
such that $\OO[F] \oplus \id a\,\omega$ is an order in $K$. In this case it is the 
unique order in $K$ of discriminant $D\id a^2$, and in particular it does not 
depend on $\omega$. We denote it by $\OO[D,\id a]$.
We say that the discriminant $(-D,\id a)$ is \emph{fundamental} if
$\OO[D,\id a]=\OO[K]$.

\begin{prop}\label{prop:Ddisc}
 
Let $\OO$ be an $\OO[F]$-order in $K$. Then there exists a unique $\id a \in 
\IF$ such that $(-D,\id a)$ is a discriminant with $\OO[D,\id a] = 
\OO$.
 
\end{prop}

\begin{proof}
 
Let $r$ be an $\OO[F]$-linear retraction for the embedding $\OO[F] 
\hookrightarrow \OO$, which we extend to a $F$-linear map $r:K \to F$. Let 
$\omega' \in K$ be any element satisfying $\disc(\omega') = -D$, and let 
$\omega =  \omega' - r(\omega')$. Then, since $\omega \notin F$, we have that 
$\ker r = F\,\omega$. Hence letting $\id a = (\OO\,\omega^{-1}) \cap F$ we have 
that $\OO = \OO[F] \oplus \id a\,\omega$. Note that the ideal $\id a$ is uniquely 
determined since $D\id a^2$ is the discriminant of $\OO$.
\end{proof}

\begin{prop}\label{prop:p*disc}
 
Let $\id p$ be a prime ideal. If $(-D,\id a)$ is a discriminant then $(-D,\id p 
\id a)$ is a discriminant, and the converse is true if $\id p \nmid 2$ and $D 
\in \id a^{-2}$.
 
\end{prop}

\begin{proof}
 
The first statement is trivial. To prove the converse, assume that $\OO[F] 
\oplus \id p\id a\,\omega$ is an order in $K$, with $\disc(\omega) = -D$. In 
particular, we have that $\trace(\omega) \in (\id p \id a)^{-1}$ and 
$\norm(\omega) \in (\id p \id a)^{-2}$. Since $\id p \nmid 2$, there 
exists $\xi \in \OO[F]$ such that $1-2\xi \in \id p$. Then, changing $\omega$ by 
$\omega - \xi \trace(\omega)$, we may assume that $\trace(\omega) \in 
\id a^{-1}$.

By hypothesis we have that $\disc(\omega) = \trace(\omega)^2 - 4 
\norm(\omega) \in \id a^{-2}$. In particular, since $\trace(\omega) 
\in \id a^{-1}$ we have that $4 \norm(\omega) \in \id a^{-2}$. Since $\id 
p \nmid 2$ and $\norm(\omega) \in (\id p \id a)^{-2}$ we have that 
$\norm(\omega) \in \id a^{-2}$, which allows us to conclude that $\OO[F] 
\oplus \id a\,\omega$ is an order in $K$.
\end{proof}

Let $(-D,\id a)$ be a discriminant, and let $\widetilde X_{D,\id a} = \{x\in 
\hatx B\,:\, \OO[D,\id a] \subseteq R_x\}$. We define a set 
$X_{D,\id a}$ of \emph{special points} associated to the discriminant $(-D,\id a)$ by
\[
 X_{D, \id a} = \hatx R\backslash\widetilde X_{D, \id a} /K^\times.
\]
If $(-D,\id a)$ is not a discriminant, we let $X_{D,\id a}=\varnothing$.
Let
\begin{equation}\notag
 \qquad
 \eta_{D, \id a} = \sum_{x \in X_{D, \id a}} 
 \frac{1}{[\OOx[x]:\OOx[F]]}\:
 \varphi_{x,P_D} \qquad \in \mathcal{M}_\mathbf{k}(R)\,,
\end{equation}
where $\OO[x] = R_x \cap K$. This is well defined by \eqref{eqn:transf_phi2} and 
\eqref{eqn:P_D-Kinv}. When $(-D,\id a)$ is fundamental then
\begin{equation}\label{eqn:eta_fund}
\eta_{D,\id a} = \tfrac1\tK\sum_{x \in X_{D, \id a}} \varphi_{x,P_D},
\end{equation}
because in this case $\OO[x]=\OO[K]$ for every $x\in X_{D,\id a}$.
It can be proved that $\eta_{D, \id a}$ does not depend on the choice of the 
embedding $K \hookrightarrow B$. When there does not exist such an embedding, we
let $\eta_{D,\id a}=0$.

\begin{prop}\label{prop:coef_serie_theta}
Let $\varphi\in \mathcal M_\mathbf{k}(R)$. Let $D\in F^+$ and let $\id a \in 
\IF$. Then the $D$-th Fourier coefficient of $\vartheta(\varphi)$ at 
the cusp $\id a$ is given by
\begin{equation}\label{eqn:fourier=height}
 \lambda(D,\id a;\vartheta(\varphi)) =
 \frac{1}{\idn a}\,\ll{\varphi, \eta_{D,\id a}}\,.
\end{equation}
\end{prop}
\begin{proof}
By \eqref{eqn:phi=sum_phix} we can assume that $\varphi = \varphi_{x,P}$ with
$P\in V_\mathbf{k}^{\Gamma_{\!x}}$, so that $\vartheta(\varphi) = 
\vartheta_{x,P}$. If $K$ does not embed into $B$, then $\mathcal{A}_{D,\id 
a}(x)=\varnothing$ and both sides of \eqref{eqn:fourier=height} vanish.

Fix $\omega \in K$ with $\disc(\omega) = -D$, and let $\Gamma_{\!x}$ act on 
$\mathcal{A}_{D,\id a}(x)$ by conjugation.
Given $y \in \mathcal{A}_{D,\id a}(x)$, since $\disc(y)=\disc(\omega)$,
we can assume there exists $\gamma\in B^\times$ such 
that $y=\gamma\omega\gamma^{-1}$. In particular $\OO[D,\id a] = \OO[F] \oplus 
\id a\,\omega$. The map $y\mapsto x\gamma$ induces an injection
\[
    \Gamma_{\!x} \backslash \mathcal{A}_{D,\id a}(x)
    \longrightarrow
    X_{D,\id a}\,.
\]
Note that
$\Stab_{\,\Gamma_{\!x}} y = (R_x\cap F(y))/\OOx[F]
\simeq
( R_{x\gamma}\cap K)/\OOx[F]=
\OOx[x\gamma]/\OOx[F]$.
Note also that
$P(y)=\ll{P\cdot\gamma,P_D}=\frac1{w_x}
\ll{\varphi_{x,P}, \varphi_{x\gamma,P_D}}$,
using that $P$ is fixed by $\Gamma_x$.
Hence
\begin{align*}
\idn a \,\lambda(D,\id a;\vartheta_{x,P})
&= 
\sum_{y\in\mathcal{A}_{D,\id a}(x)} P(y)
=
\sum_{y\in\Gamma_{\!x}\backslash\mathcal{A}_{D,\id a}(x)}
[\Gamma_{\!x}:\operatorname{Stab}_{\,\Gamma_{\!x}}y]\,
P(y)
\\&=
\sum_{x\gamma\in X_{D,\id a}}
\frac{w_x}{[\OOx[x\gamma]:\OOx[F]]}\,
\ll{P\cdot\gamma,P_D}
\\&=
\sum_{z\in X_{D,\id a}}
\frac{1}{[\OOx[z]:\OOx[F]]}\,
\ll{\varphi_{x,P}, \varphi_{z,P_D}}
=
\ll{\varphi_{x,P},\eta_{D,\id a}} \,.
\end{align*}
Note that in the last sum $\ll{\varphi_{x,P}, \varphi_{z,P_D}}=0$
unless $z=x\gamma$.
\end{proof}

By analogy with the case $F = \Q$ (see \cite{kohnen-newforms}), we consider the 
\emph{plus subspace} of $\mathcal{M}_\mathbf{3/2+k}(4\id{N}, \chi_1)$ which, 
under Hypothesis \ref{hyp:peso}, is given by
\begin{multline*}
 \mathcal{M}^+_\mathbf{3/2+k}(4\id{N}, \chi_1) = \bigl\{f \in 
\mathcal{M}_\mathbf{3/2+k}(4\id{N}, \chi_1)\,:\, \\ \lambda(D,\id a;f) = 0 
\text{ 
unless } (-D,\id a) \text{ is a discriminant}\bigr\}\,,
\end{multline*}
and we let $\mathcal{S}^+_\mathbf{3/2+k}(4\id{N}, \chi_1) = 
\mathcal{M}^+_\mathbf{3/2+k}(4\id{N}, \chi_1) \cap 
\mathcal{S}_\mathbf{3/2+k}(4\id{N}, \chi_1)$. Using the formula for the action 
of the Hecke operators in terms of Fourier coefficients (see \cite[Proposition 
5.4]{shim-hh}) together with Proposition \ref{prop:p*disc} it is easy to prove 
that $\mathcal{M}^+_\mathbf{3/2+k}(4\id{N}, \chi_1)$ is stable
by the Hecke operators $T\p$ with $\id p \nmid 2$.

\begin{coro}
 
The Hecke-linear map $\vartheta$ sends $\mathcal M_\mathbf{k}(R)$ into 
$\mathcal{M}^+_\mathbf{3/2+k}(4\id{N},\chi_1)$.
 
\end{coro}

\section{Height and geometric pairings}\label{sect:height_geom}

We start this section by comparing the geometric pairing on CM-cycles of 
\cite{zhang-gl2} (see \cite{Xue-Rankin} for the case of higher 
weight) with the height pairing introduced in Section 
\ref{sect:quat_forms}.

Let $K/F$ be a totally imaginary quadratic extension.
As in Section~\ref{sect:half_int} we assume that
there exists an embedding $K\hookrightarrow B$, which we fix.
Furthermore, we assume $\OO[K]\subseteq R$.
Let $\sC = (\hatx B / \hatx F) / (K^\times / 
F^\times)$, and let $\pi:\hatx B / \hatx F\to\sC$ be the projection 
map.
We fix a Haar measure $\mu$ on $\hatx B / \hatx F$.
On $K^\times / F^\times$ we consider the discrete measure, and we let $\bar\mu$ 
be the quotient measure on $\sC$.
We write
 $\muR = \mu\bigl(\hatx R/\hatOOx[F]\bigr)$.

We consider the space $\DC$ of 
\emph{CM-cycles} on $\sC$. These are locally constant
functions on $\sC$ with compact support.
This space comes equipped with the action of 
Hecke operators $\Tm$
given by
\begin{equation}\label{eqn:heckeC}
  \Tm\alpha(x) =
  \tfrac1{\muR}\,\int_{H_\id{m}/\hatOOx[F]}
  \alpha(hx)\,dh.
\end{equation}

Given $v\in V$ which is fixed by $K^\times/F^\times$,
we let $M_v:B^\times/F^\times\to\C$ be the matrix coefficient given by 
$\gamma\mapsto\ll{v\cdot\gamma, v}$.
Then $M_v$ is bi-$K^\times/F^\times$-invariant and satisfies that
$\overline{M_v(\gamma)} = M_v(\gamma^{-1})$.
We call $M_v$ a \emph{multiplicity function}. 
We let $k_v : \mathcal{C}\times 
\mathcal{C} \to \C$ be the map given by
\[
 k_v(x,y) = 
\sum\nolimits_{\gamma\in\Gamma'_{\!x,y}} M_v(\gamma),
\]
where for $x,y \in \hatx B$ we denote
$\Gamma'_{\!x, y} = (B^\times \cap x ^{-1} \hatx F \hatx R y) / F^\times$.
We consider the \emph{geometric pairing} 
on $\DC$ induced by $M_v$, which for $\alpha, \beta \in \DC$ that are left 
invariant by $\hatx R/\hatOOx[F]$ is given by
\begin{equation}\label{eqn:geom_pairing_def}
\ll{\alpha, \beta}_v =
\tfrac1{\muR}
\iint_{\sC\times\sC}
\alpha(x)\,\overline{\beta(y)}
\,
k_v(x,y)
\,dx\,dy\,.
\end{equation}

\begin{lemma}

Let $x,y\in \hatx B$. The natural map $\Gamma_{x,y} \to \Gamma'_{x,y}$ is 
injective, and
 \[
  \Gamma'_{x,y} = \coprod_{\xi\in \cl(F)} \Gamma_{\xi x,y}.
 \]

\end{lemma}

\begin{proof}

Let $u,v \in \hatx R$ be such that there exists $\eta \in F^\times$ with 
$x^{-1} u y \eta= x^{-1} v y$. Then $\eta = u^{-1} v \in F^\times \cap \hatx R 
= \OOx[F]$. This proves the first statement.
 
It is clear that the union gives all of $\Gamma'_{x,y}$. To see that it is 
disjoint, suppose that $\xi, \zeta\in \hatx F$ are such that there exist 
$u,v\in \hatx R$ and $\eta \in F^\times$ with $x^{-1} \xi u y \eta= x^{-1} 
\zeta v y$. Then $\xi\zeta^{-1} \eta = v u^{-1} \in \hatx F \cap \hatx R = 
\hatOOx[F]$, and hence $\xi = \zeta$ in $\cl(F)$. 
\end{proof}

The following result is immediate from this lemma and 
Proposition~\ref{prop:height_on_phi}.

\begin{prop}\label{prop:M_vs_altura}
Let $x,y\in \hatx B$. Then $k_v(x,y) = \sum_{\xi\in\cl(F)} \ll{\varphi_{\xi 
x,v}\,, \varphi_{y,v}}$.
\end{prop}

Given $a \in \hatx K$, we let
$\alpha_a \in \DC$ be the 
characteristic function of $\pi\bigl(\hatx Ra\bigr) \subseteq \sC$.
Since $\OO[K]\subseteq 
R$, the CM-cycle $\alpha_a$ depends only on the element in $\cl(K)$
determined by $a$.
The same holds for the quaternionic modular form $\varphi_{a,v}$
by \eqref{eqn:transf_phi2}.

\begin{prop}\label{prop:geom_vs_height}

Let $\id{m}\subseteq\OO[F]$ be an ideal. For $a, b \in \cl(K)$ we have that
\[
 \frac{\ll{\Tm\alpha_a, \alpha_b}_v}{\muR}
 =
 \tfrac1{\tK^{\;2}}\,
 \sum_{\xi\in\cl(F)}\ll{\Tm\varphi_{\xi a,v}, \varphi_{b,v}}\,.
\]
\end{prop}

\begin{proof}
Using \eqref{eqn:heckeC} and \eqref{eqn:geom_pairing_def}, we obtain that
\begin{align*}
    \frac{\ll{\Tm\alpha_a, \alpha_b}_v}{\muR}
& =
 \tfrac1{\muR^{\;3}}
 \iint_{\sC\times\sC}
 \int_{H_{\id m}/\hatOOx[F]}
 \alpha_a(hx)\,\alpha_b(y)\,
  k_v(x,y)
  \;dh\,dx\,dy
  \\
  & =
 \tfrac1{\muR^{\;3}}
 \iint_{\pi(\hatx Ra)\times \pi(\hatx Rb)}
 \int_{H_{\id m}/\hatOOx[F]}
  k_v(h^{-1}x,y)
  \;dh\,dx\,dy.
 \end{align*}
 Note that
 \[
 \tfrac1{\muR}
 \int_{H_{\id m}/\hatOOx[F]} k_v(h^{-1}x,y)\;dh
 =
 \sum_{h\in H_{\id m}/\hatx R}
  k_v(h^{-1}x,y)
 \]
 is constant on $\pi(\hatx Ra)\times \pi(\hatx Rb)$, and
$\muR/\bar\mu\bigl(\pi(\hatx R)\bigr)
= \vert K^\times/F^\times \cap \hatx R/\hatOOx[F]\vert
= \tK$.
\par
Using this and Proposition \ref{prop:M_vs_altura} we get that
\begin{align*}
 \frac{\ll{\Tm\alpha_a, \alpha_b}_v}{\muR}
  =
  \tfrac1{\tK^{\;2}}
  \sum_{h\in H_{\id m}/\hatx R}
    k_v(h^{-1}a,b) 
   =  
  \tfrac1{\tK^{\;2}}
  \sum_{\xi\in\cl(F)}
   \sum_{h\in H_{\id m}/\hatx R}
     \ll{\varphi_{h^{-1} \xi a,v},\varphi_{b,v}}.
 \end{align*}
Then the result follows from Proposition \ref{prop:hecke_on_phi}.
\end{proof}

Let $\alpha_K\in \DC$ be the characteristic function of $\pi(\hatx R \hatx K)$.
We have that
\begin{equation}\notag
    \alpha_K = \tfrac{\mK}{\hF} \sum_{a\in\cl(K)} \alpha_a\,.
\end{equation}
Similarly we define
\begin{align}\label{eqn:alpha_K}
    \psi_v  = \tfrac1{\tK} \sum_{a\in\cl(K)} \varphi_{a,v}\qquad \in 
\mathcal{M}_\rho(R,\bbu)\,.
\end{align}
After these definitions and Proposition~\ref{prop:geom_vs_height} we get the 
following result, analogous to \cite[Corollary 3.5]{Xue-Rankin}, 
where the author only considers the case when $F=\Q$ and $\id N$ is square-free.

\begin{coro}\label{coro:geom=height}
Let $\id{m}\subseteq\OO[F]$ be an ideal. Then
\[
 \frac{\ll{\Tm\alpha_K, \alpha_K}_v}{\muR}
 = \frac{\mK^{\;2}}{\hF}
 \,\ll{\Tm\psi_v, \psi_v}\,.
\] 
\end{coro}

\subsection*{Central values}
Let $g$ be a normalized Hilbert cuspidal newform over $F$ of level $\id N$ and 
trivial central character as in the introduction.
Write $K=F(\sqrt{-D})$ with $D\in F^+$, and denote by
$\chi_D$ the Hecke character corresponding to the extension $K/F$.
We assume that
\begin{equation}\label{eq:SigmaD}
    \Sigma_D = \mathbf{a} \cup
    \left\{\id p\mid \id N \;:\; \chi_D(\id p)^{v\p(\id N)}=-1 \right\}
\end{equation}
is of even cardinality.
For the rest of this section we let $B$ be the quaternion algebra
ramified exactly at $\Sigma_D$.
Note that this satisfies the assumption that $K$ embeds into $B$.

Let $\Tg$ be a polynomial in the Hecke operators prime to $\id N$ giving the 
$g$-isotypical projection. The following result is \cite[Theorem 
1.2]{Xue-Rankin}, which was originally proved for parallel weight $\mathbf{2}$ 
in \cite{zhang-gl2}.
\begin{thm}\label{thm:xue}
Assume $\id N\subsetneq\OO[F]$ and $\DK$ is prime to $2\id N$. Then
\begin{equation}\label{eqn:xue_orig}
 L_D(1/2,g) =
 \ll{g,g}
 \:
 \frac{\dF^{1/2}\,C(\id N)}{\mathcal N(\DK)^{1/2}}
 \:
 \frac{\llll{\Tg\alpha_K,\alpha_K}}{\muR}\,,
\end{equation}
where $C(\id N)$ is the positive rational constant given by
\[
C(\id N) = \prod_{\id p \mid \id N} (\idn p + 1) \,\idn p^{v\p(\id N) - 1}\,,
\]
and where $\llll{\:,\:}$ denotes the geometric pairing in $\DC$ given in 
\cite[(3.4)]{Xue-Rankin}.
\end{thm}

\begin{rmk}
The constant $C_1$ mentioned in \cite[Theorem 1.2]{Xue-Rankin}
contains a wrong factor, so we refer to \cite[(3.65)]{Xue-Rankin},
The constants $\mu_{\id N\,\DK}$, $\mu_{\!\varDelta^*}$ and $\mu_{\!\varDelta}$
appearing in the latter satisfy
\[
\mu_{\id N\,\DK}^{\;-1} = C(\id N\,\DK) = C(\id{N})\,C(\DK)\,,
\quad
\mu_{\!\varDelta^*} = C(\DK)\,\mu_{\!\varDelta} = 2^{\abs{S}}\,\muR\,.
\]
Using this we obtain \eqref{eqn:xue_orig}.
\end{rmk}

\begin{rmk}
The proof given in \cite{Xue-Rankin} is valid for a particular order in $B$ 
containing $\OO[K]$. Since by \cite[Proposition 3.4]{gross-local} any two orders 
in $B$ containing $\OO[K]$ are locally conjugate by an element of $\hatx K\!$, 
and the right hand side of \eqref{eqn:xue_orig} is invariant by such a 
conjugation, it follows that Theorem~\ref{thm:xue} holds for any order in $B$ containing 
$\OO[K]$.
\end{rmk}

\begin{coro}\label{coro:xue}
Under the hypotheses above, assume that $V=V_\mathbf{k}$ as in Section 
\ref{sect:half_int}, and let $P_D\in V_\mathbf{k}$ as in \eqref{eqn:gegenbauer}.
Then
\begin{equation}\notag
 L_D(1/2,g)=
 \ll{g,g}
 \:
 \frac{\dF^{\;1/2}}{\hF}
 \:
 \frac{c(\mathbf{k}) \: C(\id N)}{\mathcal N(\DK)^{1/2}}
 \:
 \frac{\mK^{\;2}}{D^\mathbf{k}}
 \:
 \ll{\Tg\psi_{P_D}, \psi_{P_D}}\,.
\end{equation}
Here $c(\mathbf{k})$ is the positive rational constant given by
$c(\mathbf{k}) = \prod_{\tau\in\mathbf{a}} \frac{r_{k_\tau}}{s_{k_\tau}}$,
where for $k\in \Z_{\geq 0}$ we denote
 \begin{equation*}
     r_k = \frac{2^{2k+1}(k!)^2}{(2k)!}\,,
 \end{equation*}
and $s_k$ is given by \eqref{eqn:s_k}.
\end{coro}
\begin{proof}
 Follows from Corollary~\ref{coro:geom=height}, Theorem~\ref{thm:xue} and the 
next lemma.
\end{proof}

\begin{lemma}
Assume that $V=V_\mathbf{k}$ as in Section \ref{sect:half_int}, 
and let $P_D\in V_\mathbf{k}$ as in \eqref{eqn:gegenbauer}.
Then
\[
 \llll{\:,\:} = \frac{c(\mathbf{k})}{D^\mathbf{k}} \, \ll{\:,\:}_{P_D}\,.
\]
 \end{lemma}

\begin{proof}
Let $M_\infty : B_\infty^\times / F_\infty^\times \to \R$ denote 
the multiplicity function considered in \cite[(3.9)]{Xue-Rankin}. Note that 
$M_{P_D}$ factors through $B_\infty^\times / F_\infty^\times$, since the 
representation $(\rho_\mathbf{k}, V_\mathbf{k})$ does. Furthermore, $M_{P_D}$ 
and $M_\infty$ are, locally, the matrix coefficient of the (up to multiplication 
by scalars) unique vector in $V_{k_\tau}$ fixed by the action of $K^\times_\tau 
/ F^\times_\tau$---the first claim follows by definition; for the second, 
see \cite[Lemma 3.13]{Xue-Rankin}. This implies $M_\infty = 
\frac{M_\infty(1)}{M_{P_D}(1)} \, M_{P_D}$.

Since $\llll{\:,\:}$ is defined in the same fashion as 
$\ll{\:,\:}_{P_D}$ but using $M_\infty$ instead of $M_{P_D}$, we have that
\[
 \llll{\:,\:} = \frac{M_\infty(1)}{M_{P_D}(1)} \,\ll{\:,\:}_{P_D} \,.
\]
Since $M_\infty(1) = \prod_{\tau \in \mathbf{a}} r_{k_\tau}$ and $M_{P_D}(1) = 
\ll{P_D,P_D}$, this together with \eqref{eqn:Mv(1)} completes the proof.
\end{proof}

\section{A result for certain orders}\label{sect:orders}

Assume in this section that $R \subseteq B$ is an order of discriminant $\id N$ 
satisfying that for every $\id p \mid \id N$ the Eichler invariant $e(R\p)$ is 
not zero. If $e(R\p) = 1$ then
\begin{equation}\label{eqn:e=1}
 R\p \simeq \left\{\left(\begin{smallmatrix} a &  b \\ \pi\p^r c & d 
\end{smallmatrix}\right) \,:\, a,b,c,d \in \OO[F\p] \right\},
\end{equation}
where $r = v\p(\id N)$. If $e(R\p) = -1$ and we let $L$ be the unique 
unramified quadratic extension of $F\p$, then
\begin{equation}\label{eqn:e=-1}
 R\p \simeq \left\{\left(\begin{smallmatrix} a  & \pi\p^ r b \\ \pi\p^{r+t} 
\overline{b} & \overline{a} \end{smallmatrix}\right) \,:\, a,b \in \OO[L] 
\right\},
\end{equation}
where $t\in \{0,1\}$ and $2r+t = v\p(\id N)$.

\begin{samepage}
\begin{prop}\label{prop:bil_local} 
 
Let $\id p$ be a prime ideal of $F$, and let $\Bil(R\p) = R\p^\times \backslash 
N(R\p) / F\p^\times$.
 
\begin{enumerate}[font=\upshape]
 \item If $\id p\nmid\id{N}$, then $\Bil(R\p)$ is the trivial group.
 \item If $\id p\mid\id{N}$, then $\Bil(R\p)$ is a group of order two generated 
by the equivalence class of an element $w\p\in R\p \cap N(R\p)$ which, in terms 
of the identifications given by \eqref{eqn:e=1} and \eqref{eqn:e=-1}, is given 
by
\[
 w\p = \begin{cases}
        \left(\begin{smallmatrix} 0  & 
\vphantom{\pi\p^r}\,\,1\,\,\\ \pi\p^{r\phantom{+t}} & 0 
\end{smallmatrix}\right), & \text{if }e(R\p) = 1, \\[2ex]
	\left(\begin{smallmatrix} 0  & \pi\p^r \\ \pi\p^{r+t} & 0 
\end{smallmatrix}\right), & \text{if }e(R\p) = -1.
       \end{cases}
\]
\end{enumerate}
 
\end{prop}
\end{samepage}

\begin{samepage}
\begin{proof} \leavevmode
 \begin{enumerate}
 \item See \cite[II.\S4, Th\'{e}or\`{e}me 2.3]{vig}.
 \item See \cite[(2.2)]{hijikata-expl} and \cite[Proposition 3]{pizer-arith2} 
for the cases $e(R\p) = 1$ and $e(R\p) = -1$ respectively. In the latter the 
author considers the case when $t = 1$, but the proof is valid in the general 
case.
\qedhere
\end{enumerate}
\end{proof}
\end{samepage}

From these local facts and \eqref{eqn:s.e.s.} we get the following statement.

\begin{prop}\label{prop:bil_charact}
 
 The group $\Bil(R)$ is isomorphic to $\prod_{\id p\mid\id{N}}\Z/2\Z$, and 
$\widetilde \Bil(R)$ is a finite group of order $\hF \, 2^{\omega(\id{N})}$.

\end{prop}

Let $D\in F^+$. Let $K = F(\sqrt{-D})$. By Proposition \ref{prop:Ddisc} there 
exists a unique $\id a \in \IF$ such that $(-D,\id a)$ is a
fundamental discriminant.
Since $\id a$ is determined by $D$, we omit it in the subindexes for the rest of 
this section.

As in Section \ref{sect:height_geom}, we assume that there exists an embedding 
$K\hookrightarrow B$ such that $\OO[K] \subseteq R$, i.e. such that $1 \in 
\widetilde X_D$. There is a left action of $\widetilde\Bil(R)$ on $X_D$, 
induced by the action of $N(\widehat R)$ on $\widetilde X_D$ by left 
multiplication. 
There is also a right action of $\cl(K) = \hatOOx[K] \backslash \hatx K / 
K^\times$ on $X_D$, induced by the action of $\hatx K$ on $\widetilde X_D$ by 
right multiplication.

\begin{lemma} Let $X_{D,\id p} = \{x\p\in B\p^\times\,:\, K\p \cap x\p^{-1} R\p 
x\p = \OO[K\p]\}$.

\begin{enumerate}[font=\upshape]
 \item The action of $\OOx[K\p] \backslash K\p^\times$ on $R\p^\times 
\backslash X_{D,\id p}$ is free.
 \item $X_{D,\id p} = N(R\p) K\p^\times$.
\end{enumerate}
 
\end{lemma}

\begin{proof} \leavevmode
\begin{enumerate}
 \item Let $a\p \in K\p^\times$ and $x\p\in X_{D,\id p}$ be such that there 
exists $u\p\in R\p^\times$ with $x\p a\p = u\p x\p$. Then $a\p = x\p^{-1} u\p 
x\p \in K\p\cap x\p^{-1} R\p^\times x\p = \OOx[K\p]$.
  
 \item Given $x\p \in X_{D,\id p}$, let $Q\p = x\p^{-1} R\p x\p$. Since $R\p$ 
and $Q\p$ contain $\OO[K\p]$ and have the same discriminant, by 
\cite[Proposition 3.4]{gross-local} there exists $a\p \in K\p^\times$ such that 
$a\p^{-1} R\p a\p = Q\p$. Then $x\p \in N(R\p) a\p$.
\qedhere
\end{enumerate}
\end{proof}

\begin{lemma}
 
Let $\id p \mid \id N$. Let $w\p\in N(R\p)$ be as in Proposition 
\ref{prop:bil_local}. If $w\p \in R\p^\times K\p^\times$, then the extension 
$K\p / F\p$ is ramified.

\end{lemma}

\begin{proof}
Write $w\p = u\p\,a\p$ with $u\p\in R\p^\times$ and $a\p$ in $K\p^\times$. Then 
$a\p\in \OO[K\p]$. Using the explicit description of $w\p$ given in 
Proposition~\ref{prop:bil_local} we see that $\pi\p \nmid a\p$ in $\OO[K\p]$. 
Furthermore, we see that  $\pi\p \mid \trace(a\p), \norm(a\p)$ in 
$\OO[F\p]$, hence $\pi\p \mid a\p^2$ in $\OO[K\p]$. Thus $\pi\p$ is 
ramified in $K\p$.
\end{proof}

\begin{samepage}
As a consequence of these lemmas and Proposition~\ref{prop:bil_local} we obtain 
the following result.
\begin{prop}\label{prop:action_transitive}
The group $\cl(K)$ acts freely on $X_D$, and the action of $\Bil(R)$ on $X_D / 
\cl(K)$ is transitive. Furthermore, the latter action is free if $(\id D_K:\id 
N) = 1$.
\end{prop}
\end{samepage}

Let $\eta_D \in \mathcal{M}_\mathbf{k}(R)$ be as in \eqref{eqn:eta_fund} and 
let $\psi_{P_D} \in \mathcal{M}_\mathbf{k}(R,\bbu)$ be as in 
\eqref{eqn:alpha_K}. 
We conclude this section by relating these quaternionic modular forms.

\begin{prop}
 
Assume that $(\id D_K:\id N) = 1$. Then
\[
 \eta_D = 
 \sum_{z \in \Bil(R)} \psi_{P_D} \cdot z\,.
\]
In particular, $\eta_D \in \mathcal{M}_\mathbf{k}(R,\mathbbm{1})^{\Bil(R)}$.
  
\end{prop}

\begin{proof}
Since $1 \in \widetilde X_D$, using \eqref{eqn:transf_phi} and 
Proposition~\ref{prop:action_transitive} we get that
\[
   \sum_{z \in \Bil(R)} \psi_{P_D} \cdot z
  = \tfrac1\tK\sum_{z \in \Bil(R)} \sum_{a \in \cl(K)} \varphi_{z^{-1} a,P_D}
  = \tfrac1\tK\sum_{x \in X_D} \varphi_{x,P_D}
  = \eta_D\,.
\]
\end{proof}

The following statement follows from this result and
Proposition~\ref{prop:bil_charact}.

\begin{coro}\label{coro:psi_vs_eta}
 Assume that $(\id D_K:\id N) = 1$. If $\varphi \in 
\mathcal{M}_\mathbf{k}(R,\mathbbm{1})^{\Bil(R)}$, then
\[
   \ll{\varphi,\eta_D}
   = 2^{\omega(\id N)}
   \ll{\varphi,\psi_{P_D}}
   \,.
\]
\end{coro}

\section{Main theorem}\label{sect:main}

Let $\mathbf{k} \in \Z_{\geq 0}^\mathbf{a}$, let $\id N \subsetneq \OO[F]$ be an 
integral ideal, and let $g\in \mathcal{S}_{\mathbf{2}+2\mathbf{k}}(\id{N},\bbu)$ 
be a normalized cuspidal newform with Atkin-Lehner eigenvalues 
$\varepsilon_g(\id p)$ for $\id p \mid \id N$, as in the introduction. Let 
$\mathscr{E}$ denote the set of functions $\varepsilon:\{\id p \,:\, \id p \mid 
\id N\} \to \{\pm 1\}$ satisfying
\begin{equation}\label{eqn:permitted_eps}
\varepsilon(\id p)^{v\p (\id N)} = \varepsilon_g(\id p) \qquad \forall \, \id p 
\mid \id N\,.
\end{equation}
Note that this set is not empty. This is equivalent to Hypothesis \ref{hyp:JL}.

Given $D \in F^+$ we let $K = F(\sqrt{-D})$, and we denote by $\chi_D$ the 
Hecke character corresponding to the extension $K / F$.
Given $\varepsilon\in\mathscr{E}$ we say that $D$ is of \emph{type
$\varepsilon$} when $\chi_D(\id p)=\varepsilon(\id p)$
for all $\id p\mid\id N$.
In particular the conductor of $\chi_D$ is  prime to $\id N$.
Hypothesis \ref{hyp:paridad} implies that for such $D$ the sign of the 
functional equation for $L_D(s,g)$ equals $1$. 

Let $B$ be the quaternion algebra over $F$ ramified exactly at
$\mathbf{a}\cup\mathcal{W}^-$, which is possible
by Hypothesis \ref{hyp:paridad}.
Fix $\varepsilon \in \mathscr{E}$, 
and let $R = R_\varepsilon\subseteq B$ be an order with discriminant $\id N$ and 
Eichler invariant
$e(R\p) = \varepsilon(\id p)$ for every $\id p \mid \id N$. Such order exists by 
\eqref{eqn:permitted_eps}, and belongs to the class of orders considered in 
Section \ref{sect:orders}.

Note that for $D$ of type $\varepsilon$
the set $\Sigma_D$ given in \eqref{eq:SigmaD}
is precisely the ramification of $B$
and moreover $\OO[K]\subseteq R$,
as required by Theorem~\ref{thm:xue}
and Corollary~\ref{coro:xue}.

Let $\pi$ be the irreducible automorphic representation of
$\GL_2$ corresponding to $g$.
For every prime $\id p$ where $B$ is ramified
$v\p(N)$ is odd by Hypothesis~\ref{hyp:JL},
hence the local component of $\pi$ at $\id p$ is square integrable.
It follows that there is an irreducible automorphic representation
$\pi_B$ of $\hatx B$
which corresponds to $\pi$ under the Jacquet-Langlands map.

In \cite[Proposition 8.6]{gross-local} it is shown that
$\hatx R$ fixes a unique line in the representation space of $\pi_B$.
This line gives an explicit quaternionic modular form
$\varphi_\varepsilon \in \mathcal{S}_\mathbf{k}(R,\bbu)$
which is well defined up to a constant.

\begin{lemma}\label{lem:sect5}
The quaternionic modular form $\varphi_\varepsilon$ is fixed by the action of 
$\Bil(R)$.
\end{lemma}

\begin{proof}
Let $\id p$ be a prime dividing $\id N$, and let $w\p \in N(R\p)$
be the generator for $\Bil(R\p)$ given in Proposition
\ref{prop:bil_local}.
Since $w\p$ has order two and normalizes $\hatx R$, it acts on
$\varphi_\varepsilon$ by multiplication by $\delta\p \in \{\pm 1\}$.

When $B$ is split at $\id p$
we have $\delta\p = \varepsilon_g(\id p)$,
and $\delta\p = -\varepsilon_g(\id p)$
when $B$ is ramified at $\id p$
(for instance, see \cite[Theorem 2.2.1]{roberts-thesis}).
Thus $\delta\p = 1$ for every $\id p\mid\id N$
by our choice of $B$,
and the result 
follows since $\{w\p\,:\,\id p\mid\id N\}$ generates $\Bil(R)$.
\end{proof}

Let $c_g$ the positive real number given by
\[
 c_g = {\ll{g,g}}\,
 \:
 \frac{\dF^{\;1/2}}{\hF}
 \:
 \frac
 {c(\mathbf{k})\:C(\id N)}
 {2^{2\omega(\id N)}}\,,
\]
where $\ll{g,g}$ is the Petersson norm of $g$,
$c(\mathbf{k})$ is as in Corollary~\ref{coro:xue}, and $C(\id N)$ is as in 
Theorem~\ref{thm:xue}.

\begin{samepage}
\begin{thm}\label{thm:main_thm}
Let $f_\varepsilon = \vartheta(\varphi_\varepsilon) \in 
\mathcal{S}^+_\mathbf{3/2+k}(4\id{N},\chi_1)$. For every $D \in F^+$
of type $\varepsilon$ such that the conductor of $\chi_D$ is prime to $2\id N$ we have
\begin{equation}\label{eqn:main_thm}
 L_D(1/2,g)=
 c_g\,
 \frac{ c_D}{D^\mathbf{k+1/2}}\,
 \frac{
 \vert \lambda(D,\id a;f_\varepsilon)\vert^2}
 {\ll{\varphi_\varepsilon,\varphi_\varepsilon}}
 \,,
\end{equation}
where $\id a\in\IF$
is the unique ideal such that $(-D,\id a)$ is a fundamental 
discriminant, $c_D$ is the positive rational number given by
$c_D=\mK^{\;2}\,\idn a$,
and $\lambda(D,\id a;f_\varepsilon)$ is the $D$-th Fourier 
coefficient of $f_\varepsilon$ at the cusp $\id a$.
\end{thm}
\end{samepage}

\begin{rmk}
 
Hypothesis \ref{hyp:paridad} implies that the sign of the functional 
equation for $L(s,g)$ equals $(-1)^\mathbf{k}$. If Hypothesis \ref{hyp:peso} 
does not hold, then both sides of \eqref{eqn:main_thm} vanish trivially, since 
$\vartheta = 0$ and $L(1/2,g) = 0$. In particular \eqref{eqn:main_thm} still 
holds, but it can not be used to compute $L(1/2,g\otimes \chi_D)$. This issue 
will be addressed in a future work by the authors.

\end{rmk}

\begin{proof}

Let $\Tg$ be the polynomial in the Hecke operators prime to $\id N$ giving the 
$g$-isotypical projection.
Let $\psi_{P_D}$ and $\eta_D$ be as in Corollary~\ref{coro:psi_vs_eta}.
Since $\Tg\psi_{P_D}$ is the $\varphi_\varepsilon$-isotypical projection of 
$\psi_{P_D}$ we have that $\Tg\psi_{P_D} 
= 
\frac{\ll{\psi_{P_D},\varphi_\varepsilon}}{\ll{\varphi_\varepsilon,
\varphi_\varepsilon }}
\, \varphi_\varepsilon$.
Combining this with Proposition~\ref{prop:coef_serie_theta}, 
Corollary~\ref{coro:psi_vs_eta} and Lemma~\ref{lem:sect5} we get that
\[
  \ll{\Tg\psi_{P_D}, \psi_{P_D}} = 
 \frac{\vert \ll{\psi_{P_D},\varphi_\varepsilon}
 \vert^2}{\ll{\varphi_\varepsilon,\varphi_\varepsilon}}
 = \frac{\vert \ll{\eta_D,\varphi_\varepsilon}
 \vert^2}{2^{2\omega(\id N)} \, 
\ll{\varphi_\varepsilon,\varphi_\varepsilon}}
 = \frac{\idn a^2}{2^{2\omega(\id N)}} \, 
\frac{\vert \lambda(D,\id 
a;f_\varepsilon)\vert^2}{\ll{\varphi_\varepsilon,\varphi_\varepsilon}}\,.
\]
Then \eqref{eqn:main_thm} follows from Corollary~\ref{coro:xue}.
\end{proof}

\begin{samepage}
\begin{coro}\label{coro:family}
 
Assume that $L(1/2,g) \neq 0$. Then $f_\varepsilon 
\neq 0$ and it maps to $g$ under the Shimura correspondence. Moreover, the set 
$\{f_\varepsilon \,:\, \varepsilon \in \mathscr{E}\}$ is linearly independent.
 
\end{coro}
\end{samepage}
 
In particular, this proves \cite[Conjecture 5.6]{yo}.
 
\begin{proof}
 
By hypotheses \ref{hyp:paridad} and \ref{hyp:peso} the sign of the functional equation for 
$L(s,g)$ equals $1$. Hence by \cite[Th\'{e}or\`{e}me 4]{waldspu-corresp} for 
every $\varepsilon \in \mathscr E$ there exists $D_\varepsilon \in F^+$ of type 
$\varepsilon$ with ${\id D}_\varepsilon = {\id D}_{F(\sqrt{-D_\varepsilon})}$
prime to $2\id N$ such that $L(1/2, g\otimes \chi_{D_\varepsilon}) \neq 
0$.
Then by \eqref{eqn:main_thm} we have that $\lambda(D_\varepsilon,\id 
a_\varepsilon;f_\varepsilon) \neq 0$, where $(-D_\varepsilon, \id 
a_\varepsilon)$ is the discriminant satisfying $D_\varepsilon \id 
a_\varepsilon^2 = \mathfrak{D}_\varepsilon$.
This, together with the Hecke-linearity of the map $\vartheta$, proves the the 
first assertion. 
The second assertion follows from the fact that if $\varepsilon' \neq 
\varepsilon$ then $\lambda(D_\varepsilon,\id a_\varepsilon;f_{\varepsilon'})=0$.
\end{proof}

We say that $D\in F^+$ is \emph{permitted}
if the conductor of $\chi_D$ is prime to $2\id N$ and 
$\chi_D(\id p)=\varepsilon_g(\id p)$
for all $\id p\mid\id N$ such that
$v\p(\id N)$ is odd.
By Hypothesis~\ref{hyp:JL},
every permitted $D$ is of type $\varepsilon$
for some $\varepsilon\in\mathscr{E}$.
\begin{coro} \label{coro:main}
There exists $f \in 
\mathcal{S}^+_\mathbf{3/2+k}(4\id{N},\chi_1)$ whose Fourier coefficients 
    satisfy
 \[
  L_D(1/2,g)= \frac{c_D}{D^\mathbf{k+1/2}} \,\vert \lambda(D,\id 
a;f)\vert^2
 \]
 for every permitted $D$, where $\id a\in\IF$ is the unique
 ideal such that $(-D,\id a)$ is a fundamental discriminant.
 Moreover, if $L(1/2,g) \neq 0$, 
then $f\neq 0$ 
and it maps to $g$ under the Shimura correspondence.
\end{coro}

In particular, this proves \cite[Conjecture 2.8]{RTV}.

\begin{proof}
 This follows from Theorem~\ref{thm:main_thm} and Corollary~\ref{coro:family}, 
 letting
 \[
    f = c_g^{1/2}\displaystyle\sum_{\varepsilon \in \mathscr{E}} 
\frac{f_\varepsilon}{\ll{\varphi_\varepsilon,\varphi_\varepsilon}^{1/2}}
    \,.
    \qedhere
 \]
\end{proof}

\end{document}